\documentclass[11pt, twoside]{amsart}

\usepackage{t1enc}
\usepackage{latexsym}
\usepackage{amssymb}
\usepackage{graphicx}
\usepackage{color}
\usepackage{amsmath}
\usepackage{amsthm}
\usepackage{amsfonts}
\usepackage{mathrsfs}
\usepackage[paper = letterpaper, left = 35mm, right = 35mm, headsep = 7mm,
footskip = 10mm, top = 30mm, bottom = 35mm, footnotesep=5mm, headheight =
2cm]{geometry}
\usepackage[all]{xy}
\usepackage[british]{babel}
\usepackage{soul}
\usepackage[hypertexnames=false,
    pdftex,
	pdfpagemode=UseNone,
	breaklinks=true,
	extension=pdf,
	colorlinks=true,
	linkcolor=blue,
	citecolor=blue,
	urlcolor=blue,
]{hyperref}
\usepackage[active]{srcltx}

\newcommand*{\widebar}{\overline}
\newcommand*{\definiere}{\mathrel{\mathop:}=}

\newcommand{\cyclic}{\mathop{\kern0.9ex{{+}\kern-2.10ex\raise-0.20
      ex\hbox{\Large\hbox{$\circlearrowright$}}}}\limits}
\newcommand{\acts}{\hspace{0.6mm}\mbox{\raisebox{0.26ex}{\tiny{$\bullet$}}}\hspace{0.6mm} }

\def\N{\ifmmode{\mathbb N}\else{$\mathbb N$}\fi}
\def\Z{\ifmmode{\mathbb Z}\else{$\mathbb Z$}\fi}
\def\Q{\ifmmode{\mathbb Q}\else{$\mathbb Q$}\fi}
\def\R{\ifmmode{\mathbb R}\else{$\mathbb R$}\fi}
\def\C{\ifmmode{\mathbb C}\else{$\mathbb C$}\fi}
\def\K{\ifmmode{\mathbb K}\else{$\mathbb K$}\fi}
\def\P{\ifmmode{\mathbb P}\else{$\mathbb P$}\fi}
\def\g{\ifmmode{\mathfrak g}\else {$\mathfrak g$}\fi}
\def\h{\ifmmode{\mathfrak h}\else {$\mathfrak h$}\fi}
\def\a{\ifmmode{\mathfrak a}\else {$\mathfrak a$}\fi}
\def\k{\ifmmode{\mathfrak k}\else {$\mathfrak k$}\fi}
\def\p{\ifmmode{\mathfrak p}\else {$\mathfrak p$}\fi}
\def\b{\ifmmode{\mathfrak b}\else {$\mathfrak b$}\fi}
\def\n{\ifmmode{\mathfrak n}\else {$\mathfrak n$}\fi}
\def\m{\ifmmode{\mathfrak m}\else {$\mathfrak m$}\fi}
\def\t{\ifmmode{\mathfrak t}\else {$\mathfrak t$}\fi}
\def\O{\ifmmode{\mathscr{O}}\else {$\mathscr{O}$}\fi}
\def\W{\ifmmode{\mathcal{V}}\else {$\mathscr{W}$}\fi}

\def\hq{\hspace{-0.5mm}/\hspace{-0.14cm}/ \hspace{-0.5mm}}
\def\kleinematrix#1,#2,#3,#4,{\begin{pmatrix}#1 & #2 \\ #3 & #4
  \end{pmatrix}}

\DeclareMathOperator{\trdeg}{trdeg}

\newtheoremstyle{daniel}{4.0mm}{0mm}{\itshape}{}{\bfseries}{.}{1.5mm}{}
\theoremstyle{daniel}
\newtheorem{thm}{Theorem}[section]
\newtheorem{prop}[thm]{Proposition}
\newtheorem{Defi}[thm]{Definition}
\newtheorem{lemma}[thm]{Lemma}
\newtheorem{cor}[thm]{Corollary}
\newtheorem{Exs}[thm]{Examples}
\newtheorem{Ex}[thm]{Example}
\newtheorem{Rems}[thm]{Remarks}
\newtheorem{Rem}[thm]{Remark}

\newtheorem*{thm*}{Theorem}
\newtheorem*{cor*}{Corollary}
\newtheorem*{prop2.2}{Proposition~2.2}

\newtheorem*{fundamentalquestion}{Fundamental Question of Geometric Invariant Theory}
\newtheorem*{prop*}{Proposition}
\newtheorem*{Notation}{Notation}

\newenvironment{rem}   {\begin{Rem}\em}{\end{Rem}}

\newenvironment{defi}  {\begin{Defi}\em}{\end{Defi}}
\newenvironment{ex}  {\begin{Ex}\em}{\end{Ex}}
\newenvironment{exs}  {\begin{Exs}\em}{\end{Exs}}

\numberwithin{equation}{section}
\def\cprime{$'$} \def\polhk#1{\setbox0=\hbox{#1}{\ooalign{\hidewidth
  \lower1.5ex\hbox{`}\hidewidth\crcr\unhbox0}}}
  \def\polhk#1{\setbox0=\hbox{#1}{\ooalign{\hidewidth
  \lower1.5ex\hbox{`}\hidewidth\crcr\unhbox0}}}
\providecommand{\bysame}{\leavevmode\hbox to3em{\hrulefill}\thinspace}
\providecommand{\MR}{\relax\ifhmode\unskip\space\fi MR }
\providecommand{\MRhref}[2]{%
  \href{http://www.ams.org/mathscinet-getitem?mr=#1}{#2}
}
\providecommand{\href}[2]{#2}

\begin{document}

\title[Complex-analytic quotients of algebraic $G$-varieties]{Complex-analytic quotients of algebraic $G$-varieties}

\author{Daniel Greb}

\address{Essener Seminar f\"ur Algebraische Geometrie und Arithmetik\\Fakult\"at f\"ur Mathe\-ma\-tik\\Universit\"at Duisburg-Essen\\
45117 Essen\\
Germany}

\email{daniel.greb@uni-due.de}{}
\urladdr{\href{http://www.esaga.uni-due.de/daniel.greb/}
{http://www.esaga.uni-due.de/daniel.greb/}}

\thanks{During the preparation of this paper, the author was partially supported by the DFG-Forschergruppe 790 ``Classification of algebraic surfaces and compact complex manifolds'', the DFG-Graduiertenkolleg 1821 ``Cohomological Methods in Geometry'', the Baden--W\"urttemberg--Stiftung through the ``Eliteprogramm f\"ur Postdoktorandinnen und Postdoktoranden'', as well as by the Institutes of Mathematics at the universities of Bonn and Freiburg.}

\date{\today}

\enlargethispage{5mm}
\maketitle
\begin{abstract} It is shown that any compact semistable quotient of a normal variety by a complex reductive Lie group $G$ (in the sense of Heinzner and Snow) is a good quotient. This reduces the investigation and classification of such complex-analytic quotients to the corresponding questions in the algebraic category. As a consequence of our main result, we show that every compact space in Nemirovski's class $\mathscr{Q}_G$ has a realisation as a good quotient, and that every complete algebraic variety in $\mathscr{Q}_G$ is unirational with finitely generated Cox ring and at worst rational singularities. In particular, every compact space in class $\mathscr{Q}_T$, where $T$ is an algebraic torus, is a toric variety. 

\end{abstract}

\section{Introduction}
Varieties and complex manifolds obtained as quotients of some simple space by an algebraic or holomorphic group action form important classes of complex spaces, which allow for explicit computation and can therefore be used as a testing ground for general ideas and conjectures. The goal of this paper is to relate several of the known constructions, with a special emphasis on algebraicity results for compact complex-analytic quotients. 

Classical examples of quotient spaces include toric varieties, which can be realised as good quotients of open subsets of representation spaces for algebraic tori by the Cox construction \cite{Cox}, flag varieties and other Fano varieties, moduli spaces parametrising representations of quivers \cite{King}, and more generally Mori dream spaces \cite{HuKeel}. Motivated by these constructions, we say that an algebraic variety (and more generally an algebraic space) \emph{belongs to class $\mathscr{Q}_G^{alg}$}, for $G$ a complex reductive Lie group, if there exists a (connected) smooth affine $G$-variety $\mathscr{X}$ and a $G$-invariant algebraic subset $\Sigma \subset \mathscr{X}$ of codimension greater than one such that the good quotient $(\mathscr{X}\setminus \Sigma)\hq G$ for the $G$-action on $\mathscr{X}\setminus \Sigma$ exists and $X \cong (\mathscr{X}\setminus \Sigma)\hq G$.

The natural analogue of a good quotient in the analytic category is the concept of an \emph{analytic Hilbert quotient} (also called \emph{semistable quotient}) for the holomorphic action of a reductive group on a complex space. This notion was introduced and studied by Snow \cite{Snow} as well as by Heinzner and his coworkers \cite{HeinznerGIT, SemistableQuotients}. It naturally appears in the study of Hamiltonian actions of reductive Lie groups on K\"ahlerian manifolds and of K\"ahlerian reduction theory, see \cite{Extensionofsymplectic, ReductionOfHamiltonianSpaces}. 

Motivated by the Levi problem on complex spaces with singularities, Nemirovski \cite{Nemirovski} recently introduced a class $\mathscr{Q}_G$ of complex spaces defined in terms of semistable quotients, which are the natural analytic analogues of the spaces in class $\mathscr{Q}_G^{alg}$ introduced above. More precisely, an (irreducible) complex space $X$ is said to \emph{belong to class $\mathscr{Q}_G$}, for $G$ a complex reductive Lie group, if there exists a (connected) Stein $G$-manifold $\mathscr{X}$, and a $G$-invariant analytic subset $\Sigma \subset \mathscr{X}$ of codimension greater than one such that the analytic Hilbert quotient $(\mathscr{X}\setminus \Sigma)\hq G$ for the $G$-action on $\mathscr{X}\setminus \Sigma$ exists and $X \cong (\mathscr{X}\setminus \Sigma)\hq G$. The interest in spaces of this form stems from the fact that questions related to the classical Levi problem are tractable for domains over complex spaces in class $\mathscr{Q}_G$. Let us give some examples: while the Levi problem is still open for unramified pseudoconvex domains over Stein spaces with non-isolated singularities, it can be solved for domains over any Stein space in class $\mathscr{Q}_G$, see \cite[Rem.~3.5]{Nemirovski}. Even in the smooth case, working in class $\mathscr{Q}_G$ has advantages: while in general even a domain in a smooth projective variety does not necessarily have a Stein envelope of holomorphy \cite[\S 2, Satz 2]{GrauertBemerkenswert}, every unramified domain over a complex space in $\mathscr{Q}_G$ admits a Stein envelope of holomorphy by \cite[Thm.~1.2]{Nemirovski}. To give another example, Ivashkovich \cite[Sect.~4.6]{Ivashkovich} recently gave a classification of locally pseudoconvex domains in Hirzebruch surfaces using the Cox realisation (i.e., considering the Hirzebruch surfaces as manifolds in $\mathscr{Q}_{(\mathbb{C}^*)^2}$).

Our first main result is the following algebraicity statement for compact spaces in class $\mathscr{Q}_G$.

\begin{thm*}[Thm.~\ref{thm:Q_G}]
 Let $G$ be a complex reductive Lie group, and let $X$ be a compact complex space in class $\mathscr{Q}_G$. Then, there exists a reductive algebraic subgroup $H < G$ such that $X \in \mathscr{Q}_H^{alg}$. More precisely, there exists a rational $H$-representation $V$ together with an $H$-invariant algebraic subset $\Theta$ of codimension greater than one such that the $H$-action on $V \setminus \Theta$ admits a good quotient $V\setminus \Theta \to Q$ with $Q^{an} \cong X$. 
\end{thm*}
In the special case where $G$ is abelian this implies the following:
\begin{cor*}[Cor.~\ref{cor:toric}]
 Let $T$ be an algebraic torus, and let $X$ be a compact complex space in class $\mathscr{Q}_T$. Then, $X$ is a toric variety. 
\end{cor*}
We show by example that for a general non-commutative group $G$ there exist examples of compact Moishezon spaces in class $\mathscr{Q}_G$ that are not algebraic varieties (see Example~\ref{ex:algebraicspaceII}). Moreover, we obtain the following rough structure theorem for \emph{varieties} in class $\mathscr{Q}_G$:
\begin{cor*}[cf.~Cor.~\ref{cor:Coxfg} and Prop.~\ref{prop:additionalproperties}] 
  Let $X$ be a complete algebraic variety in class $\mathscr{Q}_G$. Then, $X$ is unirational with at worst rational singularities, and both the divisor class group $\mathrm{Cl}(X)$ and the Cox ring $\mathrm{Cox}(X)$ are finitely generated. If $X$ is additionally assumed to be $\mathbb{Q}$-factorial and projective, then $X$ is a Mori dream space.
\end{cor*}
The aforementioned results are obtained as consequences of a much more general statement, which attacks the following

\begin{fundamentalquestion}
Given an algebraic $G$-variety or complex $G$-space $X$, find all $G$-invariant Zariski-open subsets with a good quotient, or a semistable quotient, respectively.
\end{fundamentalquestion}

In the algebraic setup, this question has been investigated and solved in a number of cases by Bia{\l}ynicki-Birula (in joint work with Sommese and {\'S}wi{\polhk{e}}cicka, see \cite{BBSurvey} for a survey) and Hausen, see for example~\cite{HausenCompleteOrbitSpaces}. The following algebraicity result reduces the fundamental question in the complex-analytic category to the corresponding question in the algebraic category, and can be summarised by saying that every \emph{compact} analytic quotient of an algebraic $G$-variety can already be obtained by the generalised Geometric Invariant Theory of Bia{\l}ynicki-Birula and {\'S}wi{\polhk{e}}cicka. 
\begin{thm}\label{mainthm}
 Let $G$ be a complex reductive Lie group, let $X$ be a $G$-irreducible normal $G$-variety, and let $\Sigma \subsetneq X$ be a $G$-invariant analytic subset such that the analytic Hilbert quotient $\pi\colon U\definiere X^{an} \setminus \Sigma \to U\hq G =:Q$ exists. Assume that $Q$ is compact. Then, the following holds.
\vspace{-2.5mm}
\begin{enumerate}
 \item[(\ref{mainthm}.1)] The complex space $Q$ is Moishezon. In particular, it is the complex space associated with an algebraic space (in the sense of Artin).
 \item[(\ref{mainthm}.2)] The set $\Sigma$ is an algebraic subvariety of $X$.
 \item[(\ref{mainthm}.3)] The map $\pi\colon U \to Q$ is a good quotient.
\end{enumerate}
\end{thm}
In Section~\ref{subsect:examples} we construct examples showing that the compactness assumption in the formulation of the previous theorem cannot be dropped. This result hence provides an essentially complete picture for the complex-analytic invariant theory of reductive algebraic group actions. It completes the investigation started in \cite{KaehlerQuotientsGIT}, where (\ref{mainthm}.2) and (\ref{mainthm}.3) were proven under the additional assumption that the quotient $Q$ be projective, cf.~\cite[Thm.~1.1]{PaHq}.

The proofs of (\ref{mainthm}.1), (\ref{mainthm}.2), and the algebraicity of $\pi$ given in Sections~\ref{subsect:1.1.1proof}, \ref{subsect:1.1.2proof}, and \ref{subsect:mapalgebraicproof}, respectively, partly build upon ideas and techniques developed in \cite{PaHq} and \cite{KaehlerQuotientsGIT}. These results generalise earlier works \cite{BBSommeseC*, BBSommeseC*2} of Bia{\l}ynicki-Birula and Sommese, in which these authors study geometric quotients of $SL_2(\mathbb{C})$-, $\mathbb{C}^*$-, and $(\mathbb{C}^*)^2$-actions.

In contrast, the proof of the central and subtle affineness property of $\pi$ given in Section~\ref{subsect:mapaffineproof} is independent of \cite{KaehlerQuotientsGIT}. It avoids the use of the equivariant version of Kodaira's embedding theorem relative to a projective quotient \cite[Thm.~3]{PaHq} and thus also provides a new and in fact more direct proof of \cite[Thm.~1.1(4)]{KaehlerQuotientsGIT}.

In the final section, it is shown that complete varieties in class $\mathscr{Q}_G$ are unirational, that they have finitely generated Cox ring and at worst rational singularities, see Corollary~\ref{cor:Coxfg} and Proposition~\ref{prop:additionalproperties}. It remains an interesting open question how to characterise spaces in class $\mathscr{Q}_G$ among spaces with these properties.

\subsubsection*{Acknowledgements}
The author wants to thank Peter Heinzner, Christian Miebach, Stefan Ne\-mi\-rovski, and Karl Oel\-je\-klaus for interesting and stimulating discussions. Furthermore, he is grateful to the organisers of the ''Russian-German conference on Several Complex Variables`` at Steklov Institute, during which some of these discussions took place, for the invitation and for their hospitality.

\section{Preliminaries}
Throughout the paper, a \emph{variety} is a reduced scheme of finite type over $\mathbb{C}$. In particular, a variety is not assumed to be irreducible. A \emph{complex space} is a reduced complex space with countable topology. Moreover, \emph{analytic subsets} of complex spaces are assumed to be closed. An \emph{open subset} refers to an open subset in the Euclidean topology of a complex space, while open subsets in the Zariski-topology will be called \emph{Zariski-open}. Our reference for the theory of algebraic spaces is \cite{KnutsonAlgebraicSpaces}.
\subsection{Good quotients and analytic Hilbert quotients}\label{sect:goodquotientsintro}
Let $G$ be a complex reductive Lie group and let $X$ be an algebraic variety equipped with an algebraic action of $G$. Then, a $G$-invariant morphism $\pi\colon X \to Y$, where $Y$ is an algebraic space (not necessarily an algebraic variety) is a \emph{good quotient} (of $X$ by the action of $G$) if 
\vspace{-2.5mm}
\begin{enumerate}
 \item $\pi$ is an affine morphism of algebraic spaces, and
 \item $(\pi_*\mathscr{O}_X)^G = \mathscr{O}_Y$ holds.
\end{enumerate}
Next we will introduce the natural analogue of this concept in the analytic category.

Let $G$ be a complex reductive Lie group and $X$ a holomorphic $G$-space. A
complex space $Y$ together with a $G$-invariant surjective holomorphic map $\pi\colon X \to Y$ is called
an \emph{analytic Hilbert quotient} or \emph{semistable quotient} of $X$ by the action of $G$ if
\vspace{-2.5mm}
\begin{enumerate}
 \item $\pi$ is a locally Stein map, and
 \item $(\pi_*\mathscr{O}_X)^G = \mathscr{O}_Y$ holds.
\end{enumerate}
\vspace{-1mm}
Here, \emph{locally Stein} means that there exists an open covering of $Y$ by open Stein subspaces
$U_\alpha$ such that $\pi^{-1}(U_\alpha)$ is a Stein subspace of $X$ for all $\alpha$. 
An analytic Hilbert quotient of a holomorphic $G$-space $X$ is unique up to biholomorphism once it exists, and we will denote it by $X\hq G$. It has the following properties (see \cite{SemistableQuotients}):
\begin{enumerate}
\item Given a $G$-invariant holomorphic map $\phi\colon X \to Z$ into a
complex
space $Z$, there exists a unique holomorphic map $\widebar \phi\colon Y \to Z$ such that
$\phi = \widebar \phi \circ \pi$.

\item For every Stein subspace $A$ of $X\hq G$ the inverse image $\pi^{-1}(A)$ is a Stein subspace
of $X$; in particular, $\pi$ is a Stein map.
\item The map $\pi$ is universal with respect to analytic subsets: If $A$ is a $G$-invariant analytic subset of $X$, $\pi(A)$ is an analytic subset of $X\hq G$, and $\pi|_A\colon A \to A\hq G$ is an analytic Hilbert quotient.
\item If $A_1$ and $A_2$ are $G$-invariant analytic subsets of $X$, then
$\pi(A_1) \cap \pi(A_2) = \pi(A_1 \cap A_2)$.
\end{enumerate}
It follows that two points $x,x' \in X$ have the same image in $X\hq G$ if and only if
$\overline{G\acts x} \cap \overline{G\acts x'} \neq \emptyset$. For each $q \in X\hq
G$, the fibre $\pi^{-1}(q)$ contains a unique closed $G$-orbit $G\acts x$. The stabiliser $G_x$ of $x$ in $G$ is a complex reductive Lie group by \cite[Prop.~2.5]{Snow} and \cite{Matsushima}. 

\begin{exs}\label{affinequotientisStein}
1.) If $X$ is a holomorphic Stein $G$-space, then the analytic Hilbert quotient exists and has the properties listed above (see \cite{HeinznerGIT} and \cite{Snow}).

2.) If $X$ is an algebraic $G$-variety with good quotient $\pi\colon X \to X\hq G$, then the associated holomorphic map of complex spaces $\pi^{an}\colon X^{an} \to (X\hq G)^{an}$ is an analytic Hilbert quotient, see \cite{Lunaalgebraicanalytic} as well as \cite[Sect.~6.4]{HeinznerGIT}. Note that while these references discuss the case of quotients in the category of algebraic varieties, the general case can be reduced to this one by a base change diagram. 
\end{exs}

\subsection{A criterion for the existence of good quotients}
In our study it will be important to decide whether a given $G$-variety admits a good quotient. For this, we will use the following criterion established by Bia{\l}ynicki-Birula and {\'S}wi{\polhk{e}}cicka:
\begin{thm}[Theorem B of \cite{BBExistenceOfQuotients}]\label{thm:BBcriterion}
Let $X$ be a normal algebraic variety equipped with an action of the reductive group $G$. If there exists a $G$-invariant affine morphism $f\colon X \to Z$ into an algebraic space $Z$, then there exists a good quotient $X \to X \hq G$.
\end{thm}
The assumption that $f$ is affine is crucial for the proof of the above result, which has recently been generalised by Alper and Easton \cite[Thm.~1.1]{AlperEaston} to the non-normal setup. Consequently, in order to apply Theorem~\ref{thm:BBcriterion} it is important to detect affine morphisms. In our context the following observation turns out to be central. For a proof, the reader is referred to \cite[Prop.~I.4.12 and Extension II.3.8]{KnutsonAlgebraicSpaces}; cf.~also \cite[Prop.~2.7.1.xiii]{EGAIVPartII}.
\begin{lemma}[Testing affineness on \'etale coverings]\label{lem:etaleaffine}
 Let $\pi\colon X \to Y$ be a morphism from an algebraic variety $X$ to an algebraic space $Y$. Let $\hat Y_\alpha$, $\alpha \in I$, be an algebraic variety, and let $\{\eta_\alpha \colon \hat Y_\alpha \to Y\}$ be an \'etale covering of $Y$. Let $\hat\pi_\alpha \colon \hat X_\alpha \definiere \hat Y_\alpha \times_Y X \to \hat Y_\alpha$ be the base change of $\pi$ from $Y$ to $\hat Y_\alpha$. Then, in order for $\pi$ to be affine it suffices that $\hat \pi_\alpha$ be an affine morphism of schemes for every $\alpha \in I$.
\end{lemma}

\section{Proof of the main result}

The following notation will be in use for the whole section: let $G$ be a complex reductive Lie group, let $X$ be a $G$-irreducible normal  algebraic $G$-variety, and let $U \subset X$ be a non-empty analytically Zariski-open $G$-invariant open subset with complement $\Sigma$ such that the analytic Hilbert quotient $\pi\colon U \to Q$ exists. We suppose that $Q$ is compact.

\subsection{Reduction to the case $G$ connected}
Assume that Theorem~\ref{mainthm} is proven for the case of connected groups. Then, taking first the analytic Hilbert quotient by the connected component of the identity $G^0$ of $G$, which exists by \cite[Sect.~2, Prop.~1]{SemistableQuotients}, and then the analytic Hilbert quotient by the finite group $\Gamma \definiere G/G^0$, we obtain the following commutative diagram:
$$\begin{xymatrix}{
 X &\ar[l] U\ar[rd]_{\pi} \ar[r]^{\pi_{G^0}}& \ar[d]^{\pi_\Gamma}U \hq G^0\\
   &   & U\hq G.
}
  \end{xymatrix}
 $$
Since $X$ was assumed to be normal, its irreducible components are mutually disjoint, the decomposition of $U$ into irreducible components is given by intersecting the irreducible components of $X$ with $U$, and consequently, $U\hq G^0$ consists of a number of disjoint (compact) irreducible components, each arising as the quotient of one of the irreducible components of $U$ by the action of $G^0$. 
Going through the statement of Theorem~\ref{mainthm} applied to $\pi_{G^0}$ (or rather to its restriction to any of the irreducible components of $U$), we see that  $U \hq G^0$ is a (possibly disconnected) algebraic space, that the intersection of $\Sigma$ with any connected component of $X$ is algebraic, and hence $\Sigma$ itself is algebraic, and that $\pi_{G^0}\colon U \to U\hq G^0$ is a good quotient. Since good quotients for actions of finite groups always exists in the category of algebraic spaces, cf.~\cite[Thm.~4.3.2]{BBSurvey}, the quotient $Q = (U\hq G^0)/\Gamma$ is an algebraic space and the map $\pi = \pi_{\Gamma} \circ \pi_{G^0}$ is a good quotient. 

Consequently, without loss of generality we may assume that $G$ is connected, which we do from now on. Accordingly, from now on we will assume that $X$ is irreducible.

\subsection{Proof of (\ref{mainthm}.1): $X$ is Moishezon}\label{subsect:1.1.1proof}
The aim of this section is to show that $X$ is Moishezon, and hence the complex space associated with a complete algebraic space, cf.~\cite[Thm.~7.3]{ArtinAlgebraizationII}. The sheaf of germs of meromorphic functions on a complex space $Z$ will be denoted by $\mathscr{M}_Z$. Recall that for any irreducible algebraic $G$-variety $X$ the subset
\begin{equation}\label{eq:X_gen}
 X_{\mathrm{gen}}\definiere \{x \in X \mid \dim G\acts x = m_X\},
\end{equation}
where $m_X = \max_{x \in X} \{\dim G\acts x\}$, is $G$-invariant, Zariski-open, and dense in $X$.

We formulate the main result of this section as a separate theorem:

\begin{thm}
 Let $G$ be a connected complex reductive Lie group, let $X$ be an irreducible normal $G$-variety, and let $\Sigma \subsetneq X$ be a $G$-invariant analytic subset such that the analytic Hilbert quotient $\pi\colon U\definiere X^{an} \setminus \Sigma \to U\hq G =:Q$ exists. Assume that $Q$ is compact. Then, $Q$ is Moishezon.
\end{thm}

\begin{proof}
An application of \cite[Lemma~6.2]{KaehlerQuotientsGIT} yields a $G$-stable irreducible algebraic subvariety $Y$ of $X$ such that $\pi(Y \cap U) = Q$ and such that the intersection $Y_{\mathrm{gen}} \cap U$ contains an orbit that is closed in $U$. A dimension count then shows that
\begin{equation}\label{eq:Qdim}
 \dim Q = \dim Y - m_Y.
 \end{equation}
By Rosenlicht's Theorem \cite{Rosenlicht2}, we have 
\begin{equation}\label{eq:Rosenlicht}
 \trdeg_\mathbb{C}(\mathbb{C}(Y)^G) = \dim Y -m_Y.
\end{equation}
Moreover, since the generic $G$-orbit in $U\cap Y$ is closed in $U$, it follows from \cite[Cor.~5.3]{KaehlerQuotientsGIT} that invariant meromorphic functions on $Y\cap U$ descend to $Q$; in other words, we have field monomorphisms 
\begin{equation}\label{eq:monomo}
 \mathbb{C}(Y)^G \hookrightarrow \mathscr{M}_Y(Y\cap U)^G \hookrightarrow \mathscr{M}_Q(Q).
\end{equation}
Consequently, we obtain the following chain of (in)equalities
\begin{align*}
\dim Q     &=   \dim Y -m_Y &&\text{by \eqref{eq:Qdim}}\\
           &= \trdeg_\C \C(Y)^G   &&\text{by \eqref{eq:Rosenlicht}}\\
           &\leq \trdeg_\C\mathscr{M}_Q(Q) &&\text{by \eqref{eq:monomo}}.
\end{align*}
Hence, $\trdeg_\C\mathscr{M}_Q(Q) = \dim Q$, i.e., $Q$ is Moishezon, as claimed.
\end{proof}

\subsection{Proof of (\ref{mainthm}.2): $\Sigma$ is algebraic}\label{subsect:1.1.2proof}
We formulate the main result of this section as a separate theorem:
\begin{thm}\label{thm:subvarietyalgebraic}
Let $G$ be a connected complex reductive Lie group, let $X$ be an irreducible normal algebraic $G$-variety, and let $U = X \setminus \Sigma$ be a non-empty $G$-invariant analytically Zariski-open subset of $X$ such that the analytic Hilbert quotient $\pi\colon U \to Q$ exists. If $Q$ is (the complex space associated with) a complete algebraic space, then $U$ is Zariski-open, and $\Sigma$ is an algebraic subvariety of $X$.
\end{thm}
Under the assumption that $Q$ is a complete algebraic variety, the result was shown in \cite[Thm.~6.1]{KaehlerQuotientsGIT}. We closely follow the arguments given there, carefully indicating in which way these have to be adapted in order to apply also in case of non-schematic quotients. 

\emph{Rosenlicht subsets}, i.e., Zariski-open subsets of a given algebraic $G$-variety $X$ with a geometric quotient having the field of invariant rational functions on $X$ as its function field, play a prominent role, and we refer the reader to \cite[Sect.~3.1]{KaehlerQuotientsGIT} for a thorough discussion.

Moreover, motivated by a result of Sumihiro \cite{completion}, if $G$ is a connected algebraic group and $X$ an algebraic $G$-variety, a $G$-invariant, Zariski-open, quasi-projective neighbourhood of $x$ in $X$ that can be $G$-equivariantly embedded as a Zariski locally closed subset of the projective space $\mathbb{P}(V)$ associated with some rational $G$-representation $\rho\colon G \to \mathrm{GL}(V)$ is called a \emph{Sumihiro neighbourhood} of $x$ in $X$. 

Furthermore, the following notations will be used in the proof: If $X$ is a holomorphic $G$-space and $A$ is a $G$-stable subset of $X$, then we set \[\mathcal{S}_G^X(A) \definiere \{ x \in X \mid \overline{G \acts x} \cap A \neq \emptyset\},\]and call it the \emph{saturation} of $A$ with respect to the $G$-action on $X$. Moreover, given $X$ and $U$ as in Theorem~\ref{mainthm} and a $G$-invariant analytic subset $A$ of $X$, we set $A^{ss}\definiere A \cap U$.	

We start by looking at the case where the generic $G$-orbit is closed in $X$, cf.~\cite[Sect.~6.2]{KaehlerQuotientsGIT}.
\begin{prop}\label{prop:rosenlichtandsemistable}
Let $G$ be a connected complex reductive Lie group, let $X$ be an irreducible algebraic $G$-variety, and let $U \subset X$ be a $G$-invariant analytically Zariski-open subset of $X$ such that the analytic Hilbert quotient $\pi\colon U \to Q$ exists. Assume that $Q$ is a complete algebraic space and that $U \cap X_{\mathrm{gen}}$ contains an orbit that is closed in $U$.

If $V_R$ is any Rosenlicht subset of $X$, then $U\cap V_R$ contains a $G$-invariant Zariski-open (Rosenlicht) subset $U_R$ of $V_R$ consisting of $G$-orbits that are closed in $U$ such that there exists an open algebraic embedding $\imath\colon U_R/G \hookrightarrow Q$ making the following diagram commutative:
\[\begin{xymatrix}{
  &\ar_{\pi}[ld]  U_R  \ar^{p}[dr]&  \\
Q &     & \ar@{_{(}->}_{\imath}[ll] U_R/G.}
\end{xymatrix}
\]
In particular, the quotient map $\pi\colon U \to Q$ extends to a rational map $X \dasharrow Q$, whose restriction to $U$ is a morphism of algebraic spaces.
\end{prop}
\begin{proof}
 Let $V_R$ be any Rosenlicht set of $X$ and let $p: V_R \to V_R/G$ be the quotient map. We denote the complement of $V_R$ in $X$ by $V_R^c$. Without loss of generality we can assume that $V_R/G$ is affine. Our first aim is to show that $V_R/G$ is birational to $Q$. By \cite[Cor.~5.3]{KaehlerQuotientsGIT}, the restriction $f|_U$ of every $G$-invariant meromorphic function $f\in \mathscr{M}_X(X)^G$ to $U$ descends to a meromorphic function on $Q$. As a consequence of the Chow lemma for complete algebraic spaces, see \cite[Thm.~IV.3.1]{KnutsonAlgebraicSpaces} or \cite[Thm.~7.3]{ArtinAlgebraizationII}, these meromorphic functions are in fact rational. Consider the rational map $\varphi: Q \dasharrow V_R/G$ that corresponds to $p|_U$. Since elements of $\varphi^*(\C(V_R/G)) = \C(X)^G \subset \C(Q)$ separate orbits in $V_R$, they separate orbits in
\[U_R \definiere U \setminus \mathcal{S}^U_G(V^c_R \cap U) \subset U\cap V_R.\]

This set is non-empty, analytically Zariski-open, and $\pi$-saturated in $U$. Since $U$ contains a closed orbit of generic orbit dimension and since the algebraic space $Q$ is complete, $U_R$ is mapped to a non-empty Zariski-open subset of $Q$, again as a consequence of the Chow lemma or GAGA for complete algebraic spaces. Therefore, $\varphi$ is generically one-to-one and hence birational. Let $\imath \definiere \varphi^{-1} : V_R/G \dasharrow Q$. Without loss of generality, we may assume that $V_R/G$ coincides with the set where $\imath$ is an isomorphism onto its image. It follows that $\imath: V_R/G \hookrightarrow Q$ is an open embedding. We obtain the following commutative diagram
\[\begin{xymatrix}{
U_R  \ar@{^{(}->}[rr]\ar[d]_{p|_{U_R}}&  &U\ar[d]^\pi \\
p(U_R) \ar@{^{(}->}[rr]^<<<<<<<<<<{\imath|_{p(U_R)}}&  &Q.
}
  \end{xymatrix}
\]
Since $Q$ is complete, the image $p(U_R)$ is Zariski-open in $Q$. It follows that $U_R$ is Zariski-open in $V_R$ and contained in $U$. An  application of Lemma~\ref{lem:holomorphicrational} concludes the proof.
\end{proof}

We now return to the general case. The following is the analogue of \cite[Lem.~6.4]{KaehlerQuotientsGIT} in our more general setup.
\begin{lemma}\label{lem:UcontainsZopen}
Let $G$ be a connected complex reductive Lie group, $X$ an irreducible $G$-variety, and $U$ a non-empty  $G$-invariant analytically Zariski-open subset of $X$ such that the analytic Hilbert quotient $\pi: U \to Q$ exist and such that additionally $Q$ is a complete algebraic space.
Assume that every point $x\in U$ whose orbit $G\acts x$ is closed in $U$ has a Sumihiro neighbourhood. Then, $U$ contains a non-empty $G$-invariant Zariski-open subset of $X$.
\end{lemma}
\begin{proof}
By \cite[Lem.~6.2]{KaehlerQuotientsGIT}, there exists an irreducible $G$-invariant subvariety $Y$ of $X$ such that $Y^{ss}\hq G = Q$ and such that $Y^{ss} \cap Y_{\mathrm{gen}}$ contains an orbit that is closed in $U$. By Proposition~\ref{prop:rosenlichtandsemistable} there exists a $G$-invariant Zariski-open subset $A$ of $Y$ contained in $U$ such that $\pi(A)$ is an open subset of $Q$. Furthermore, $A$ consists of
orbits which are closed in $U$.

Let $W$ be an irreducible Sumihiro neighbourhood of a point $x \in A$ and let $\psi: W \to  \P(V)$ be a $G$-equivariant embedding of $W$ into the projective
space associated with a rational $G$-representation $V$. Let $Z$ be the closure of $\psi(W)$ in $\P(V)$. Since $Z$ is projective, it has a Rosenlicht quotient constructed from the Chow variety of $Z$, see \cite[Sect.~3.1]{KaehlerQuotientsGIT}. Let $U_R$ be such a Rosenlicht subset of $Z$ having the properties stated in Proposition~3.1 of \cite{KaehlerQuotientsGIT}. Then, Lemma 7.3 of \cite{PaHq} implies that $\mathcal{S}^Z_G(\psi(A \cap W))\cap U_R$ is constructible in
$U_R$. Therefore,
$\mathcal{S}_G^X(A \cap W)\cap W$ contains a $G$-invariant Zariski-open subset $\widetilde U$ of its closure. By construction, $\pi(A \cap W)$ is open in $Q$, and hence, $\pi^{-1}(\pi(A \cap W)) \cap W$ is an open subset of $X$ that is contained in
$\mathcal{S}_G^X(A \cap W)\cap W$. We conclude that $\widetilde{U}$ is Zariski-open in
$X$.
\end{proof}
With these preparations at hand, we are now in the position to prove the following openness result.
\begin{prop}\label{prop:SumihiroZopen}
Let $G$ be a connected complex reductive Lie group, $X$ a $G$-variety, and let $U$ be a non-empty $G$-invariant analytically Zariski-open subset of $X$ such that the analytic Hilbert quotient $\pi: U \to Q$ exist and such that additionally $Q$ is a complete algebraic space.
Assume that every point $x \in U$ whose orbit $G\acts x$ is closed in $U$ has a Sumihiro neighbourhood. Then, $U$ is Zariski-open in $X$.
\end{prop}
\begin{proof}
Let $X = \bigcup_j X_j$ be the decomposition of $X$ into irreducible components. Then, there exists a $j_0 \in \{1, \dots, m\}$ such that $X_{j_0}^{ss}$ is analytically Zariski-open in $X_{j_0}$ and non-empty. Let $V_{j_0}$ be the $G$-invariant Zariski-open subset of $X_{j_0}^{ss}$ whose existence is
guaranteed by the previous Lemma. Set $\widetilde {X} \definiere X \setminus \bigl(V_{j_0} \setminus \bigcup_{j\neq j_0} X_j \bigr)$. Then either $\widetilde{X} = U^c \subset X$ or $\widetilde X^{ss} \neq \emptyset$. In the first case, we are done, since $\widetilde {X}$ is algebraic in $X$. In the second case, we notice that $\widetilde X^{ss}$ is analytically Zariski-open and $G$-invariant in $\widetilde X$, that $\widetilde X^{ss} \hq G = \pi (\widetilde X^{ss}) \subset Q$ is a complete algebraic space by \cite[IV.Thm.~3.1]{KnutsonAlgebraicSpaces}, and that the existence of Sumihiro neighbourhoods is inherited by $\widetilde X^{ss}$. We proceed by Noetherian induction.
\end{proof}
\begin{proof}[Proof of Theorem~\ref{thm:subvarietyalgebraic}]
Using Proposition~\ref{prop:SumihiroZopen} it suffices to note that, since $X$ is assumed to be normal and $G$ to be connected, every point in $X$ has a Sumihiro neighbourhood by the result \cite{completion} which motivated the coining of the term.
\end{proof}

\begin{rem}\label{rem:affineeasier}
 If $X$ is an affine variety (as in the application to spaces in class $\mathscr{Q}_G$, cf.~Section~\ref{sect:Q_G}), the proof is actually much easier: Since $Q$ is compact, we obviously have $\mathscr{O}_Q(Q) = \mathbb{C}$. Because of the natural morphisms
$$\mathscr{O}_X(X)^G \hookrightarrow \mathscr{O}_U(U)^{G} \overset{\cong}{\longrightarrow} \mathscr{O}_Q(Q),$$
$X$ does not admit non-constant invariant holomorphic functions. As a consequence, the good quotient $X\hq G = \{\mathrm{pt.}\}$ (which exists, since $X$ is assumed to be affine) is a singleton. An application of the $G$-equivariant version of Chow's Lemma \cite[Thm.~10.3]{PaHq} (which in this special case also follows easily from \cite[Sect.~5, Satz]{HeinznerFixpunktmengen}) then yields that $\Sigma$ is an algebraic subvariety of $X$.
\end{rem}

\subsection{Algebraicity of the quotient map}\label{subsect:mapalgebraicproof}
We formulate the main result of this section as a separate theorem:
\begin{thm}\label{thm:quotientalgebraic}
Let $G$ be a connected complex reductive Lie group, let $X$ be an irreducible normal algebraic $G$-variety, and let $U$ be a $G$-invariant Zariski-open subset of $X$ such that the analytic Hilbert quotient $\pi\colon U \to Q$ exists. If $Q$ is (the complex space associated with) a complete algebraic space, then the quotient map $\pi$ is algebraic. 
\end{thm}

The following lemma implies that in order to proof Theorem~\ref{thm:quotientalgebraic}, it suffices to show that there exists a Zariski-open subset $W \subset U$ such that $\pi|_W\colon W \to Q$ is a morphism of algebraic spaces. While this is certainly well-known, we provide a proof for the reader's convenience.

\begin{lemma}[Regularity of rational holomorphic maps]\label{lem:holomorphicrational}
 Let $U$ be a Zariski-open subset of an irreducible normal algebraic variety $X$, $Y$ an irreducible complete algebraic space, and $f\colon U \to Y$ a morphism of algebraic spaces. If $f^{an}\colon U^{an} \to X^{an}$ extends to a holomorphic map $\overline{f^{an}}\colon X^{an} \to Y^{an}$, then $f $ extends to a morphism $\bar f$ with $(\bar f)^{an} = \overline{f^{an}}$.
\end{lemma}
\begin{proof}
 Let $\Gamma_f \subset X \times Y$ be the Zariski-closure of the graph of $f$ with the two projections $\bar f \definiere \mathrm{pr}_2|_{\Gamma_f}\colon \Gamma_f \to Y$ and $\pi := \mathrm{pr}_1|_{\Gamma_f}\colon \Gamma_f \to X$. The claim of the lemma is equivalent to $\pi$ being biregular. 

As a first observation, note that $(\Gamma_f)^{an}$ is the graph of $\overline{f^{an}}$. It follows that $\pi$ is one-to-one, hence quasi-finite. On the other hand, as $Y$ is complete, $\pi$ is proper, and therefore finite. In particular, $\pi$ is an affine morphism. Consequently, it follows from \cite[Cor.~II.6.16]{KnutsonAlgebraicSpaces} that $\Gamma_f$ is a scheme.  Now, as $(\Gamma_f)^{an}$ is the graph of $\overline{f^{an}}$, the map $\pi^{an}$ is biholomorphic. Let $$\eta:= (\pi^{an})^{-1}\colon X^{an} \to (\Gamma_f)^{an}$$ be its holomorphic inverse, and note that this is a map between (the complex spaces associated with) two algebraic varieties. As $f\colon U \to Y$ is a morphism of algebraic spaces, $\eta$ is not only holomorphic, but also induced by a rational map from $X$ to $\Gamma_f$, and hence regular by \cite[Lem.~in VIII.3]{ShafarevichII}.
\end{proof}

With these preparations in place, the reasoning of \cite[Sections~7.1 and 7.2]{KaehlerQuotientsGIT} now applies mutatis mutandis (replacing the use of GAGA for complete algebraic varieties \cite{GAGA} by GAGA for complete algebraic spaces (\cite[Thm.~7.3]{ArtinAlgebraizationII} or \cite[Thm.~IV.3.1]{KnutsonAlgebraicSpaces}) and \cite[Prop.~6.3]{KaehlerQuotientsGIT} by Proposition~\ref{prop:rosenlichtandsemistable}). This concludes the proof of Theorem~\ref{thm:quotientalgebraic}. \hfill \qed

\subsection{The quotient map is affine}\label{subsect:mapaffineproof}
In general, it is rather delicate to decide whether a given categorical quotient map is affine, as the following well-known example \cite[Ex.~0.4]{MumfordGIT} shows.
\begin{ex}\label{ex:notaffine}
Let $X = \mathbb{C}^2 \setminus \{(0,0)\}$ be endowed with the standard action of $SL_2(\mathbb{C})$ (coming from its defining representation. Let $T$ be the diagonal maximal torus in $SL_2(\mathbb{C})$, and $N < SL_2(\mathbb{C})$ be its normaliser. The action of $N$ on $\mathbb{C}^2$ has a good quotient $\hat \pi\colon \mathbb{C}^2 \to \mathbb{C} =: Q$, given by $\pi\colon \mathbb{C}^2 \to \mathbb{C}, (z,w) \mapsto (zw)^2$. The restriction $\pi\definiere \hat\pi|_{X}\colon X  \to Q$ is a geometric categorical quotient for the action of $N$ on $X$. Since $X = \mathbb{C}^2\setminus \{(0,0)\}$ is not affine, $\pi$ is not a good quotient. This example can be made into an example for a connected group by considering $SL_2(\mathbb{C}) \times_N X \to Q$, which is a geometric quotient for the induced action of $SL_2(\mathbb{C})$.
\end{ex}
Note that at the same time that one proves that the quotient map in Example~\ref{ex:notaffine} is not affine, one proves that it is not (locally) Stein. Theorem~\ref{thm:affine} below can be interpreted as saying that for quotients by reductive group actions there are not more obstructions to being affine than there are for being Stein. For general morphisms of algebraic varieties this is of course far from being true, as classical examples \cite[Chap.~6, \S 3]{HartshorneAmple} of non-affine quasi-projective Stein varieties show.

\begin{thm}\label{thm:affine}
 Let $G$ be a connected complex reductive Lie group, let $X$ be an irreducible normal algebraic variety, and let $U$ be a $G$-invariant Zariski-open subset of $X$ with analytic Hilbert quotient $\pi\colon U \to Q$. Assume that
\begin{enumerate}
 \item[(a)]  $Q$ is (the complex space associated with) a complete algebraic space, and
 \item[(b)] $\pi$ is a morphism of algebraic spaces.
\end{enumerate}
  Then, $\pi$ is an affine map.
\end{thm}
The main ingredient in the proof of Theorem~\ref{thm:affine} is the following \'etale slice theorem of Luna type, cf.~\cite{LunaSlice}.
\begin{prop}\label{prop:LunaSlice}
Let $G$ be a connected complex reductive Lie group, let $U$ be a normal irreducible algebraic $G$-variety with analytic Hilbert quotient $\pi \colon U \to Q$. Assume that
\begin{enumerate}
 \item[(a)] $Q$ is (the complex space associated with) a complete algebraic space, and
 \item[(b)] $\pi\colon U \to Q$ is a morphism of algebraic varieties.
\end{enumerate}
 Then, for every point $q \in Q$ there exists a normal irreducible affine $G$-variety $T$ and a $G$-equivariant morphism $\varphi: T \to U$ such that $\varphi$ and the induced morphism $\overline{\varphi}\colon T\hq G \to Q$ have the following properties:
\begin{enumerate}
 \item the image of $\varphi$ is a Zariski-open $\pi$-saturated neighbourhood of $\pi^{-1}(q)$,
 \item both $\varphi$ and $\overline{\varphi}$ are \'etale,
 \item the induced morphism $T \to T\hq G \times_Q U$ is biregular. 
\end{enumerate}
\end{prop}
\begin{rem}\label{rem1}
 In the setup of Proposition~\ref{prop:LunaSlice} it follows in particular that the morphism $T \hq G \times_Q U \to T\hq G$ (the base change of $\pi$ from $Q$ to $T\hq G$) is affine.
\end{rem}
Morphisms satisfying properties $(1) - (3)$ of Proposition~\ref{prop:LunaSlice} are often called \emph{excellent}, e.g.~see \cite[p.~209]{PopovVinberg}.

\begin{rem}
 Note that traditionally Luna's Slice Theorem \cite{LunaSlice} is established \emph{assuming the existence} of a good quotient $X \to Q$. Here, we reverse the approach and (using the existence of an analytic Hilbert quotient) establish the slice theorem first, from which the existence of a good quotient will follow; cf.~the strategy pursued in \cite{HeinznerGIT}.
\end{rem}
\vspace{-0.3cm}

\begin{proof}[Proof of Proposition~\ref{prop:LunaSlice}]
Let $p \in \pi^{-1}(q)$ such that $G\acts p$ is the uniquely determined closed $G$-orbit in $\pi^{-1}(q)$. Since the isotropy group $H \definiere G_p$ of $p$ is reductive, cf.~Section~\ref{sect:goodquotientsintro}, it follows as in the proof of \cite[Thm.~5.6]{PaHq} that there exists a Zariski-locally closed $H$-stable affine subvariety $S$ of $U$ with the following property: if $\varphi \colon G \times_H S \to U$ is the natural induced $G$-equivariant morphism, and $\pi_T\colon G \times_H S \to T\hq G \cong S\hq H$  denotes the good quotient of the affine $G$-variety $G\times_HS$ by $G$, then there exists a $\pi_T$-saturated open neighbourhood $\widetilde V$ of $[e,p]$ in $G\times_H S$ such that 
\begin{equation}\label{eq:text}
\text{$\varphi|_{\widetilde V}$ is biholomorphic, and $\varphi(\widetilde V)$ is a $\pi$-saturated open neighbhd.~of $G\acts p$ in $U$.} 
\end{equation}
We set $T\definiere G \times_HS$ and summarise our situation in the following commutative diagram:
\begin{equation}\label{eq:diagram1}
\begin{gathered}
\begin{xymatrix}{
     T \ar[r]^{\varphi}\ar[d]_{\pi_T} &  U\ar[d]^\pi  \\
   T\hq G \ar[r]^{\overline{\varphi}}& Q.
}
  \end{xymatrix}
\end{gathered}
\end{equation}
We note that $\pi$ maps $G$-invariant Zariski-closed subsets of $U$ to Zariski-closed subsets of $Q$: indeed, if $A$ is any $G$-invariant algebraic subset of $U$, the image $\pi(A) \subset Q$ is both constructible (by Chevalley's Theorem), and an analytic subset in $Q$ (since $\pi$ is universal with respect to $G$-invariant analytic subsets of $U$, cf.~Section~\ref{sect:goodquotientsintro}), and therefore Zariski-closed. 

The previous observation allows us to make the following reduction steps: Shrinking $S$ if necessary, we may assume that $T$ is normal and irreducible, and that $\varphi(T) \subset U$ is $\pi$-saturated and Zariski-open in $U$. Since both $\varphi$ and $\overline{\varphi}$ are \'etale near $G\acts [e, p]$, and $\pi_S([e,p])$, respectively, shrinking $T$ further, we may additionally assume that both $\varphi$ and $\overline{\varphi}$ are \'etale. 

We are aiming to show that after shrinking $T$ the map $\varphi$ maps closed $G$-orbits to closed $G$-orbits. 

Let $\mathscr{C}_T$ and $\mathscr{C}_U$ denote the set of closed orbits in $T$ and $U$, respectively. Since $\pi_T\colon~T \to T\hq G$ is a good quotient, it follows from the Luna Slice Theorem that $\mathscr{C}_T$ is a constructible subset of $T$. Moreover, we claim that $\mathscr{C}_U$ is likewise constructible in $U$. In order to see this, let $\{S_\gamma\}_{\gamma \in I}$ be the Luna stratification of $Q$ as an analytic Hilbert quotient, see \cite[Sect.~1.2]{ReductionInStepsBook} or \cite[Sect.~2]{Extensionofsymplectic}. Then, as $Q$ is a complete algebraic space, $\{S_\gamma\}_{\gamma \in I}$ is a stratification by Zariski-locally closed subsets of $Q$ (this is the only place where our argument uses the assumption on completeness of $Q$, which can probably be avoided by a finer analysis of the situation). Consequently, each inverse image $\pi^{-1}(S_\gamma)$ is Zariski-locally closed in $U$. The closed orbits in $\pi^{-1}(S_\gamma)$ are exactly the orbits of minimal dimension in $\pi^{-1}(S_\gamma)$. It follows that $\mathscr{C}_U \cap \pi^{-1}(S_\gamma)$, which coincides with the set of closed orbits in $\pi^{-1}(S_\gamma)$, is Zariski-closed in $\pi^{-1}(S_\gamma)$. We deduce that $\mathscr{C}_U$ is constructible, as claimed. We note that both $\mathscr{C}_T$ and $\mathscr{C}_U$ are connected, as the corresponding quotients $T\hq G$ and $Q$ are connected. 

Since $\varphi$ is an algebraic morphism, $\varphi(\mathscr{C}_T) \cap \mathscr{C}_U$ is a constructible subset of $\mathscr{C}_U$. As locally near $[e, p] \in T$ in the Euclidean topology, the morphism $\varphi$ maps closed $G$-orbits to closed $G$-orbits by \eqref{eq:text}, the intersection $\varphi(\mathscr{C}_T) \cap \mathscr{C}_U$ contains a $G$-invariant open neighbourhood of $\varphi(p)$ in $\mathscr{C}_U$ and therefore a $G$-invariant Zariski-open neighbourhood of $p$ in $\mathscr{C}_U$. Consequently, there exists a Zariski-open $G$-invariant neighbourhood $U_\mathscr{C}$ of $[e, p]$ in $\mathscr{C}_T$ that is mapped into $\mathscr{C}_U$ by $\varphi$. The saturation $\mathcal{S}_G^T(U_\mathscr{C})$ is a Zariski-open and $\pi_T$-saturated neighbourhood of $p$ in $T$, and by construction every closed $G$-orbit in $\mathcal{S}_G^T(U_\mathscr{C})$ is mapped to a closed $G$-orbit in $U$. We may therefore assume that $T$ itself has this property, which we will do from now on.

Next we consider the natural map $\psi\colon T \to T\hq G \times_Q U, t \mapsto (\pi_T(t), \varphi(t))$, sitting inside the following extended version of diagram~\eqref{eq:diagram1}:
\[ 
\xymatrix{
T \ar@/_/[ddr]_{\pi_T} \ar@/^/[drr]^\varphi \ar@{->}[dr]|-{\psi} \\
   & T\hq G \times_Q U \ar[d]_<<<{\mathrm{pr}_1} \ar[r]_<<<<{\mathrm{pr}_2} & U \ar[d]^{\pi}\\
   & T\hq G \ar[r]^{\overline{\varphi}}& Q. }
\]

Note that $T\hq G \times_Q U$ is a scheme; as we will see below, it is actually normal and in particular reduced. We claim that $\psi$ is surjective. Let $(\pi_T(t), u) \in T\hq G \times_Q U$ be any point. Without loss of generality, we may assume that the $G$-orbit of $t$ is closed in $T$. It suffices to show the $\pi_T$-fibre $F_t$ through $t$ surjects via $\varphi$ onto the $\pi$-fibre $F_u$ through $u$. For this, we notice that on the one hand, as $\varphi$ maps closed orbits to closed orbits, $\varphi(G \acts t)$ coincides with the closed $G$-orbit $G\acts u'$ in $F_u$ (which might be different from $G\acts u$) and on the other hand, as $\varphi$ is \'etale and hence open, a small $K$-invariant Euclidean neighbourhood $N_T$ of $G\acts t$ is mapped onto a $K$-invariant open neighbourhood $N_U$ of $G\acts u'$ in $U$. Moreover, since the closure of every $G$-orbit in $F_u$ contains $G\acts u'$, and since $\varphi(T)$ is $G$-invariant, we have $F_u \subset \varphi(G\acts N_T)$. As $\bar \varphi$ is \'etale, possibly after shrinking $N_T$, we know that $\varphi^{-1}(F_u) \cap G\acts N_T \subset F_t$. We conclude that $\varphi(F_t) = F_u$.

Surjectivity of $\psi$ and irreducibility of $T$ together imply that the fibre product is irreducible. Since $\overline{\varphi}$ is \'etale, $\mathrm{pr}_2$ is \'etale. This has the following consequences:
\begin{itemize}
 \item Since $U \subset X$ is  normal, the fibre product $T\hq G \times_Q U$ is likewise normal.
 \item Since $\varphi$ can be factored as $\varphi = \mathrm{pr}_2\circ \psi$, the map $\psi$ is quasi-finite.
\end{itemize}
We claim that the degree $\deg(\psi)$ of the quasi-finite morphism $\psi$ is equal to one, which will imply that $\psi$ is birational. Indeed, let $\widetilde V$ be as in \eqref{eq:text} above, let $V \definiere \pi_T(\widetilde V)$, $\varphi(\widetilde V) \definiere \widetilde W$, $\overline{\varphi}(V) = \pi(\widetilde W) =: W$, and denote the inverse mapping of $\varphi|_{\widetilde V}$ by $\eta\colon \widetilde W \to \widetilde V$. We note that $\mathrm{pr}_1^{-1}(V) = V \times_W \widetilde W$, and that $\psi|_{\widetilde V}\colon \widetilde V \to \mathrm{pr}_1^{-1}(V)$ is biholomorphic with holomorphic inverse given by $(\pi_T(t), x) \mapsto \eta (x)$. It follows that for any point $z$ in the open subset $\mathrm{pr}_1^{-1}(V) \subset T\hq G \times_Q U$ we have $\mathrm{deg}(\psi) = |\psi^{-1}(z)| = 1$.

Summing up, we have proven that $\psi$ is a surjective birational quasi-finite map from $T$ to the normal irreducible variety $T\hq G \times_Q U$. Zariski's Main Theorem \cite[p.209, Original form]{RedBook} hence implies that $\psi$ is biregular, as claimed.
\end{proof}

\begin{proof}[Proof of Theorem~\ref{thm:affine}] 
Let $G$, $X$, and $\pi\colon U \to Q$ be as in the statement of Theorem~\ref{thm:affine}. Let $q \in Q$ be given, and let $\varphi\colon T \to U$ and $\bar\varphi: T\hq G \to Q$ be as in Proposition~\ref{prop:LunaSlice}. Then, $\bar\varphi\colon T\hq G \to Q$ gives an affine \'etale neighbourhood of $q$ in $Q$ such that the induced base change $T\hq G \times_Q U \to T\hq G$ of $\pi$ is an affine morphism of varieties, see Remark~\ref{rem1}. Hence, the algebraic space $Q$ has a covering in the \'etale topology such that the base change of $\pi$ to any open set of the covering is an affine map (of algebraic varieties). Consequently, $\pi$ itself is affine by Lemma~\ref{lem:etaleaffine}.
 \end{proof}

\subsection{Proof of (\ref{mainthm}.3), end of proof of Theorem~\ref{mainthm}}\label{subsect:proofaffine}
Since the map $\pi\colon U \to Q$ is a $G$-invariant affine morphism of algebraic spaces, we are in the position to apply Theorem~\ref{thm:BBcriterion} to conclude that there exists a good quotient $p\colon U  \to \hat Q$. Since the induced map $p^{an}\colon U^{an} \to \hat Q^{an}$ of complex spaces is an analytic Hilbert quotient for the action of $G$, see Example~\ref{affinequotientisStein}.2, there exists a biholomorphic map $\psi\colon Q \to \hat Q^{an}$, which by GAGA \cite[Thm.~7.3]{ArtinAlgebraizationII} is induced by an isomorphism of algebraic spaces, making the following diagram commutative:
\[\begin{xymatrix}{
   &  U \ar[rd]^{p^{an}} \ar[ld]_\pi&   \\
 Q   \ar[rr]_{\psi}^{\cong} &   &  \hat Q^{an}.
}
  \end{xymatrix}
\]
Hence, $\pi$ is (the holomorphic map associated with) a good quotient. This concludes the proof of Theorem~\ref{mainthm}. \hfill \qed

\subsection{Examples}\label{subsect:examples}
In this section we discuss two examples which show in which sense the statement of Theorem~\ref{mainthm} is optimal.

While good quotients of normal algebraic varieties by complex tori are (normal) algebraic varieties, see \cite[Cor.~7.1.3]{BBSurvey}, the following example shows that, if the group $G$ is not commutative, one cannot expect every quotient $Q$ as in Theorem~\ref{mainthm} to be an algebraic variety; hence it is necessary to work in the category of algebraic spaces.

\begin{ex}\label{ex:algebraicspaceI}
Let $G_{ss} \definiere SL_2(\mathbb{C})$, and let $$V = \{\text{forms of degree } 5 \text{ in two variables}\}$$ be the unique irreducible $SL_2(\mathbb{C})$-representation of dimension $6$. The Grassmannian $Y \definiere \mathrm{Gr}(3,6)$ of $3$-planes in $V$ naturally inherits a $G_{ss}$-action from $V$. The Pl\"ucker embedding $\mathrm{Gr}(3,6) \hookrightarrow \mathbb{P}(\bigwedge^3 V)$ is $G_{ss}$-equivariant and projectively normal. Consequently, the affine cone $X \definiere \mathrm{Cone}(Y) \subset \bigwedge^3V$ over $Y$ is a normal affine algebraic variety, endowed with an algebraic action of $G = \mathbb{C}^* \times G_{ss} = \mathbb{C}^* \times SL_2(\mathbb{C})$. By \cite[Ex.~on p.~15]{BBSCompleteSL2OrbitSpaces} there exists a Zariski-open $G_{ss}$-stable subset $U_Y \subset Y$ such that the good (geometric) quotient $U_Y \to Q = U_Y / G_{ss}$ exists, and such that $Q$ is a complete algebraic space, but not an algebraic variety. If $p\colon X \setminus \{0\} \to Y$ denotes the orbit projection for the $\mathbb{C}^*$-action, set $U \definiere p^{-1}(U_Y)$. Then, $U$ is a $G$-invariant Zariski-open subset of $X$, the good (geometric) quotient $\pi\colon U \to U / G = Q$ exists, and $Q$ is (by construction) a complete algebraic space, but not an algebraic variety.
\end{ex}

The following example shows that without a compactness assumption on the quotient the statements of Theorem~\ref{mainthm} do not continue to hold, not even in case $X = V$ a representation and $\Sigma$ of codimension at least two.
\begin{ex}
Let $X = \mathbb{C} \times \mathbb{C}^2$, endowed with the algebraic $\mathbb{C}^*$-action on the first factor. Let $\Gamma$ be a lattice of rank $4$ in $\mathbb{C}^2$. Set $\Sigma = \mathbb{C} \times \Gamma$. Then, the analytic Hilbert quotient $\pi\colon X\setminus \Sigma \to Q$ of $X \setminus \Sigma$ exists, and $Q$ is biholomorphic to $\mathbb{C}^2 \setminus \Gamma$. Since the fundamental group of $\mathbb{C}^2\setminus \Gamma$ is not finitely generated, $Q$ cannot be endowed with the structure of an algebraic space (of finite type).  
\end{ex}

By construction, the open subset $X\setminus \Sigma$ considered in the previous example is not $\mathbb{C}^*$-maximal, i.e., not maximal with respect to $\mathbb{C}^*$-saturated inclusion, cf.~\cite{BBSurvey}. In fact, the $\mathbb{C}^*$-maximal subset $X$ contains $X\setminus \Sigma$ as a $\mathbb{C}^*$-saturated subset, and $X$ has a good quotient whose restriction to $X\setminus \Sigma$ coincides with $\pi$. Hence, one might hope that Theorem~\ref{mainthm} continues to hold at least for $G$-maximal subsets $U \subset X$, although many of the methods used in our argument are not applicable in this more general situation.

\section{Spaces in class $\mathscr{Q}_G$}\label{sect:Q_G}
\subsection{Nemirovski's class $\mathscr{Q}_G$}
Recently, Nemirovski \cite{Nemirovski} introduced a class of complex spaces in which questions related to the classical Levi problem are tractable. We quickly recall his construction:

An (irreducible) complex space $X$ is said to \emph{belong to class $\mathscr{Q}_G$}, for $G$ a complex reductive Lie group, if there exists a (connected) Stein $G$-manifold $\mathscr{X}$, and a $G$-invariant analytic subset $\Sigma \subset \mathscr{X}$ with $\mathrm{codim}_{\mathscr{X}}(\Sigma)> 1$ such that the analytic Hilbert quotient for the $G$-action on $\mathscr{X}\setminus \Sigma$ exists and $X \cong (\mathscr{X}\setminus \Sigma)\hq G$.

\begin{rem}
 The codimension condition on $\Sigma$ is actually unnecessary: Assume that $\mathscr{X}$ is a connected Stein $G$-manifold, and $\Sigma \subset \mathscr{X}$ a $G$-invariant analytic subset such that $X \cong (\mathscr{X}\setminus \Sigma)\hq G$. Let $\Sigma = \Sigma_1 \cup \Sigma_2$ 
be the decomposition of $\Sigma$ into its codimension one part and the remaining components (all of which have codimension greater than or equal to two). Note that both $\Sigma_1$ and $\Sigma_2$ are $G$-invariant. In particular, $\mathscr{X}\setminus \Sigma_1$ is a $G$-invariant open subset of $\mathscr{X}$. Since $\mathscr{X}$ is smooth, the ideal sheaf $\mathscr{I}_{\Sigma_1}$ is locally principal, and we may apply \cite[V.\S 1.1, Thm.~5]{TheoryOfSteinSpaces} to conclude that $\mathscr{X}'\definiere \mathscr{X}\setminus \Sigma_1$ is Stein. Obviously, $\mathscr{X}\setminus \Sigma = \mathscr{X}' \setminus \Sigma_2$ and hence $(\mathscr{X}' \setminus \Sigma_2) \hq G = X$ is in class $\mathscr{Q}_G$.
\end{rem}
A similar reasoning applies in the algebraic case. We hence define an (irreducible) algebraic space $X$ to \emph{belong to class $\mathscr{Q}_G^{alg}$}, for $G$ a complex reductive Lie group, if there exists a (connected) affine algebraic $G$-manifold $\mathscr{X}$, and a $G$-invariant algebraic subset $\Sigma \subset \mathscr{X}$ (such that $\mathrm{codim}_{\mathscr{X}}(\Sigma)> 1$) such that the good quotient for the $G$-action on $\mathscr{X}\setminus \Sigma$ exists and $X \cong (\mathscr{X}\setminus \Sigma)\hq G$.

Examples are provided by analytic Hilbert quotients $\mathscr{X} \hq G$, toric varieties (for $G$ the product of an algebraic torus with a finite group, using the Cox realisation \cite{Cox}), flag manifolds (for $G$ a product of $GL(n_k)$'s), and moduli spaces of semistable quiver representations \cite{King}. 

In this note, we focus on compact spaces in class $\mathscr{Q}_G$. The following rough structure result (and more) was already observed in \cite[Sect.~2.4]{Nemirovski}.
\begin{prop}\label{prop:roughstructure}
 Let $X$ be a complex space in class $\mathscr{Q}_G$ with $\mathscr{O}_X(X) = \mathbb{C}$. Then, the Stein $G$-manifold $\mathscr{X}$ fulfills $\mathscr{O}_{\mathscr{X}}(\mathscr{X})^G = \mathbb{C}$. As a consequence, $\mathscr{X}$ carries the structure of an affine algebraic $G$-variety, uniquely determined up to $G$-equivariant biregular isomorphism.
\end{prop}

\subsection{Consequences of the main result}
We apply our main result to the setup of class $\mathscr{Q}_G$-spaces.
\begin{thm}\label{thm:Q_G}
Let $G$ be a complex reductive Lie group, and let $X$ be a compact complex space in class $\mathscr{Q}_G$. Then, there exists a reductive algebraic subgroup $H < G$ such that $X \in \mathscr{Q}_H^{alg}$. More precisely, there exists a rational $H$-representation $V$ together with an $H$-invariant algebraic subset $\Theta$ of codimension greater than one such that the $H$-action on $V \setminus \Theta$ admits a good quotient $V\setminus \Theta \to Q$ with $Q^{an} \cong X$. 
\end{thm}
\begin{proof}Let $X = (\mathscr{X}\setminus \Sigma) \hq G$ be in class $\mathscr{Q}_G$. By Proposition~\ref{prop:roughstructure} above, $\mathscr{X}$ is an affine algebraic $G$-variety without non-constant invariant regular functions. It follows from the main result, Theorem~\ref{mainthm}, that $X$ is Moishezon, that $\Sigma$ is algebraic, and that $\mathscr{X}\setminus \Sigma$ admits a good quotient $\mathcal{P}$ by $G$ whose associated complex space is biholomorphic to $X$. Moreover, as already observed by Nemirovski \cite[Sect.~2.4]{Nemirovski}, it follows from Luna's slice theorem \cite{LunaSlice} that there exists a reductive algebraic subgroup $H$ of $G$ and a rational $H$-representation $V$ such that $\mathscr{X}$ is $G$-equivariantly biregular to $G \times_H V$. Intersecting $\Sigma$ with the fibre of $G \times_H V \to G/H$ over $eH$, we obtain an $H$-invariant algebraic subset $\Theta$ of $V$ with $\mathrm{codim}_V \Theta > 1$ such that the good quotient $Q = (V \setminus \Theta)\hq H$ exists and is biregular to $\mathcal{P}$. This concludes the proof.
\end{proof}
\begin{cor}\label{cor:toric}
 Let $T$ be an algebraic torus, and let $X$ be a compact complex space in class $\mathscr{Q}_T$. Then, $X$ is a toric variety. 
\end{cor}
\begin{proof}
 By Theorem~\ref{thm:Q_G}, there exists a closed algebraic subgroup $H < T$ and a rational $H$-representation $V$ together with a $H$-invariant Zariski-open subset $U$ that admits a good quotient $p\colon U \to Q$ with $Q^{an} \cong X$. Being an algebraic subgroup of an algebraic torus, the group $H$ splits as a direct product $H = \Gamma \times T'$, where $\Gamma$ is a finite group, and $T'$ is an algebraic torus, see for example \cite[Cor.~on p.~114]{OnishchikVinberg}. As $H$ is a direct product, there are two ways of taking the quotient of $U$ by $H$ in steps. Doing this, we obtain the following commutative diagram of good quotients:
$$\begin{xymatrix}{
 U  \ar[rd]^{p}\ar[r]^{p_{T'}}\ar[d]_{p_\Gamma}&  U \hq T'\ar[d]^{q_{\Gamma}}   \\
  U/\Gamma \ar[r]_{q_{T'}} & Q. 
}
  \end{xymatrix}
 $$
The quotient $U/\Gamma$ is a Zariski-open subset of the affine variety $V/\Gamma$, hence a normal quasi-affine variety. As a good quotient of this normal algebraic variety by the algebraic torus $T'$, the algebraic space $Q$ is a normal algebraic variety, for example by \cite[Cor.~7.1.3]{BBSurvey}. On the other hand, as $q_\Gamma$ is finite, the quotient $U\hq T'$ is complete, which implies that $U$ is a $T'$-maximal subset of $V$. It hence follows from \cite[Thm.~1.6]{BBRecipeForVectorSpaces} and \cite[Cor.~6.1]{BBOpenSubsetsOfProjectiveSpacesWithQuotient} that  $U\hq T'$ is a torus embedding of $\hat T := \mathbb{T}/ T'$, where $\mathbb{T}$ is a ''big`` torus in $PGL_\mathbb{C}(V \oplus \mathbb{C})$ containing the image of $H$ (note that the action of $H$ on $V$ is diagonalisable). As $\mathbb{T}$ is commutative, the action of $\Gamma$ commutes with the action of $\hat T$ on $U\hq T'$. Consequently, the $\hat T$-action on $U\hq T'$ descends to an algebraic $\hat T$-action on $Q$. As $U\hq T'$ is toric for $\hat T$, the action of $\hat T$ on $Q$ has an open orbit. Hence, we conclude that $Q$ is a torus embedding for a quotient of $\hat T$. This concludes the proof.
\end{proof}
In view of Example~\ref{ex:algebraicspaceI} and of the result just proven, one might ask whether every compact space in class $\mathscr{Q}_G$, where $G$ is some complex reductive Lie group, is a complete algebraic variety, and not just a complete algebraic space, i.e., whether an additional smoothness assumption implies that the quotient is a variety. The next example shows that for general non-commutative reductive groups $G$ the answer to this question is negative. 
\begin{ex}\label{ex:algebraicspaceII}
We continue to use the notation of Example~\ref{ex:algebraicspaceI}. This time, we will exploit the fact that the Grassmannian $Y$ is in class $\mathscr{Q}_{GL_3(\mathbb{C})}$. More precisely, $Y$ is the quotient of the Zariski-open subset $W_{\mathrm{max}}$ of maps of maximal rank in $W := \mathrm{Hom}_{\mathbb{C}}( \mathbb{C}^3, V)$ by the natural $GL_3(\mathbb{C})$-action (by precomposition); let $q: W_{\mathrm{max}} \to Y$ be the quotient map. Moreover, the $G_{ss}$-action on $Y$ described in Example~\ref{ex:algebraicspaceI} lifts to the natural $G_{ss}$-action on $W$ (by postcomposition). As before, let $U_Y$ be the Zariski-open $G_{\mathrm{ss}}$-stable subset of $Y$ with good (geometric) quotient $U_Y \to Q = U_Y / G_{ss}$, where $Q$ is a complete algebraic space, but not an algebraic variety. Then, $U':= q^{-1}(U) \subset W$ is a $(GL_{3}(\mathbb{C}) \times G_{ss})$-invariant Zariski-open subset of $W$, and the good (geometric) quotient $U' \to U'/ (GL_{3}(\mathbb{C}) \times G_{ss}) \cong Q$ exists. Hence, $Q$ is in class $\mathscr{Q}_{GL_{3}(\mathbb{C}) \times G_{ss}}$, and $Q$ is a complete algebraic space  but not an algebraic variety.
\end{ex}

\begin{rem}
 Note that in the previous example the group $GL_{3}(\mathbb{C}) \times G_{ss}$ is neither commutative nor semisimple, in accordance with Corollary~\ref{cor:toric} and with the following observation: let $G$ be a semisimple complex Lie group and let $V$ be a $G$-module, such that there exists a non-empty Zariski-open subset $U \subset V$ with complete quotient $U \hq G$. Note that $V$ is a factorial affine algebraic variety and that factoriality is preserved by passing to an open subset $V' \supset U$ with codimension one complement, see \cite[Chap.~II, Props. 6.2 and 6.5]{Hartshorne}. As observed before, we have $\mathscr{O}_{V'}(V')^G = \mathbb{C}$. Now, classical results concerning invariant rational functions on factorial affine $G$-varieties, see e.g.~\cite[Thm.~3.3]{PopovVinberg}, imply that $G$ has an open orbit in $V'$, and hence that $Q$ is a point; in particular, $Q$ is an algebraic variety.
\end{rem}

\subsection{Varieties with finitely generated Cox ring}
While the motivation to study spaces in class $\mathscr{Q}_G$ comes from the classical Levi problem, varieties with finitely generated Cox ring and Mori dream spaces form an important class of higher-dimensional varieties for which many of the features of the Minimal Model Program are well-understood. We quickly recall their definition.
\begin{defi}\label{defi:MDS}
 Let $X$ be an irreducible complete normal algebraic variety with only constant invertible global regular functions and finitely generated divisor class group $\mathrm{Cl}(X)$. The \emph{Cox-ring} (or \emph{total coordinate ring}) of $X$ morally is defined to be
$$\mathrm{Cox}(X) = \bigoplus_{[D] \in \mathrm{Cl}(X)} H^0\bigl(X, \, \mathscr{O}_X(D) \bigr).$$
Due to the potential existence of torsion in $\mathrm{Cl}(X)$ and due to the question how to identify section spaces of linearly equivalent divisors, the precise definition requires some care, see for example~\cite[Sect.~2]{Baeker}. A $\mathbb{Q}$-factorial projective variety with finitely generated Cox ring is called a \emph{Mori Dream Space}.
\end{defi}
\begin{rem}\label{rem:MDSalaHuKeel}
Note that in \cite{HuKeel} a Mori Dream Space is defined to be a $\Q$-factorial projective variety $X$ such that 
\begin{itemize}
 \item $\mathrm{Pic}(X)_\Q = N^1(X)_\Q$, and
 \item for some choice of $\Q$-basis $\{L_1, \dots, L_k\}$ of $N^1(X)_\Q$ whose affine hull contains the cone of effective divisors, the associated ring
\[\bigoplus_{(\nu_1, \dots, \nu_k) \in \N^k} H^0 \bigl(X, L_1^{\nu_1} \otimes \dots \otimes L_k^{\nu_k}\bigr)\]
 is finitely generated.
\end{itemize}
We note that any Mori Dream Space in the sense of Hu-Keel is also a Mori Dream Space as defined in Definition~\ref{defi:MDS}. For more information the reader is referred to \cite[Thm.~2.37]{HausenThreeLectures}.
\end{rem}

As a consequence of our main result we obtain:
\begin{cor}\label{cor:Coxfg}
 Let $X$ be a complete algebraic variety in class $\mathscr{Q}_G$. Then, both the divisor class group $\mathrm{Cl}(X)$ and the Cox ring $\mathrm{Cox}(X)$ are finitely generated. If $X$ is additionally assumed to be $\mathbb{Q}$-factorial and projective, then $X$ is a Mori dream space.
\end{cor}
\begin{proof}
 As we have seen above, there exists a reductive algebraic subgroup $H$ of $G$ as well as a rational $H$-representation $V$ together with an $H$-invariant algebraic subset $\Sigma$ of codimension greater than one such that the $H$-action on $V \setminus \Sigma$ admits a good quotient $V\setminus \Sigma \to Q$ with $Q^{an} \cong X$. Since the Cox ring $\mathrm{Cox}(V)\cong \mathbb{C}[V]$ is finitely generated and since $\mathscr{O}_X(X)^* = \mathbb{C}^*$, an application of \cite{Baeker} yields the desired result.
\end{proof}

\begin{rem}
Non-$\Q$-factorial and non-projective toric varieties exist in abundance, cf.~\cite[Prop.~4.2.7 and Ex.~4.2.13]{CoxLittleSchenck}. As complete toric varieties are in $\mathscr{Q}_T$ for $T$ the product of some algebraic torus and a finite group, see \cite[Thm.~5.1.11]{CoxLittleSchenck}, not every variety in $\mathscr{Q}_G$ is $\Q$-factorial or projective.
\end{rem}
Based on the observations made above, it is a natural question to ask whether there exists an intrinsic characterisation of those Mori dream spaces or varieties with finitely generated Cox ring that are in class $\mathscr{Q}_G$ for some complex reductive group $G$. The following result yields two necessary criteria, the first one local and the second one of global nature.

\begin{prop}\label{prop:additionalproperties}
 Let $X$ be a complex space in class $\mathscr{Q}_G$. Then, 
\vspace{-0.2cm}
\begin{enumerate}
 \item $X$ has at worst reductive quotient singularities. In particular, $X$ has rational singularities.
 \item $X$ is unirational.
\end{enumerate}
\end{prop}
Here, we say that a complex space $X$ has \emph{reductive quotient singularities}, if every point $p \in X$ has open neighbourhood $U$ biholomorphic to an open neighbourhood of $\pi(0) \in W\hq H$, where $H$ is a complex reductive Lie group, $W$ is a rational $H$-representation, and $\pi: W \to W\hq H$ is the invariant-theoretic quotient map. 

\begin{proof}[Proof of Proposition~\ref{prop:additionalproperties}]
The first part is an immediate consequence of the holomorphic slice theorem \cite[p.~80]{Snow} and the main result of \cite{GrebAnalyticSingularities}. For the second part, it suffices to note that owing to its algebraicity the quotient morphism $\pi\colon V \supset U \to X$ extends to a dominant rational map $\mathbb{P}(V \oplus \mathbb{C}) \dasharrow X$.
\end{proof}

We note that in general, varieties with finitely generated Cox ring and even Mori dream spaces can have worse than rational singularities. The following example was constructed in joint discussions with Patrick Graf and Alex K\"uronya. The author would like to thank these colleagues for allowing him to include it here. 
\begin{ex}\label{ex:MDS}The example is mainly an exercise in computing on projective cones. We will use additive notation both for divisors and their associated divisorial sheaves. 

Let $Y \subset \mathbb{P}^N$ be a smooth positive-dimensional variety, and let $X$ be the projective cone over $Y$. We assume $Y$ to be projectively normal, so that $X$ is normal. If $\widetilde X$ is the blowup of the vertex $P \in X$, we obtain a commutative diagram
\[ \xymatrix{
\widetilde X \ar[dr]_\pi \ar[rr]^f & & X \ar@{-->}[dl]^p \\
& Y &
} \]
and $\widetilde X = \mathbb{P}_{\mathrm{sub}}(\mathscr{O}_Y(1) \oplus \mathscr{O}_Y)$ is a $\mathbb{P}^1$-bundle over $Y$. Write $L = \mathscr{O}_Y(1)$, and let $E \cong Y$ be the exceptional divisor of the resolution $f$. Then, the normal bundle $E|_E$ of $E$ in $\widetilde X$ is $-L$. The projection $\pi$ has another section, which does not meet $E$. Its image under $f$ will be denoted by $S_\infty$.

Let $D$ be a divisor on $Y$. Then $p^* D$ is defined by pulling back $D$ on the locus where $p$ is a morphism, and then taking the closure in $X$. If $D$ is any divisor on $\widetilde X$, then the pushdown $f_*D$ is defined to be the Weil divisor $f(D)$ on $X$. Note that $f_*$ preserves linear equivalence of Weil divisors.

Our example will be a projective cone over a suitable hypersurface in some $\mathbb{P}^N$. In order to establish the Mori dream space property, we will proof a number of preliminary claims.

\noindent \textbf{Claim 1:} \emph{The canonical divisor of $X$ is given by $K_X = p^*(K_Y - L)$.}

\noindent\emph{Proof.} First, we calculate the canonical divisor of $\widetilde X$. The Picard group of $\widetilde X$ is generated by $\pi^*(\mathrm{Pic}(Y))$ and $E$. So we may write
\[ K_{\widetilde X} = \pi^*(M) + \lambda E \quad\quad \text{for some } M \in \mathrm{Pic}(Y), \lambda \in \mathbb{Z}.\]
Let $F \cong \P^1$ be a fibre of $\pi$. Then $K_F = K_{\widetilde X}\big|_F$, because the normal bundle of $F$ in $\widetilde X$ is trivial. Since $E \cdot F = 1$, we get $\lambda = -2$. On the other hand, by the adjunction formula for $E \subset \widetilde X$, we obtain \[ M - 2 E|_E = K_{\widetilde X}\big|_E = K_E - E|_E, \]
so $M = K_E + E|_E = K_Y - L$. In summary, we arrive at $K_{\widetilde X} = \pi^*(K_Y - L) - 2 E$. Claim 1 follows from this by passing through $f$. \hfill\qed

\noindent \textbf{Claim 2:}  \emph{Let $D$ be a divisor on $Y$. Then $p^* D$ is $\Q$-Cartier if and only if there is a $\lambda \in \Q$ such that $D \sim_\Q \lambda L$. (``$D$ is $\Q$-linearly proportional to $L$.'')}

\noindent\emph{Proof.} 
 See \cite[Ex.~3.5]{HaconKovacs}.\hfill \qed

\noindent \textbf{Claim 3:}  \emph{The pair $(X, \emptyset)$ is log canonical if and only if $K_Y \sim_\Q \lambda L$ for some $\lambda \in \Q^{\le 0}$.
 }

\noindent\emph{Proof.} By Claims 1 and 2, the variety $X$ is $\Q$-Gorenstein if and only if $K_Y \sim_\Q \lambda L$ for some $\lambda \in \Q$. In this case, the ramification formula for $f$ reads
\[ K_{\tilde X} = f^* K_X + aE \quad \text{for some }a \in \mathbb{Q}. \]
Restricting to $E$ and using the adjunction formula for $E \subset \tilde X$, we get $K_E - E|_E = aE|_E$, so $K_Y = -(a + 1)L$.
Consequently, $a \ge -1$ is equivalent to $\lambda \le 0$.\hfill \qed

\noindent \textbf{Claim 4:}  \emph{Any (Weil) divisor $D$ on $X$ is linearly equivalent to $p^* D'$ for some $D' \in \mathrm{Pic}(Y)$.
 }

\noindent\emph{Proof.} 
The strict transform $f^{-1}_* D$ can be written as $f^{-1}_* D \sim \pi^* D' + \lambda E$.
Pushing this down by $f$ gives $D \sim f_* \pi^* D' = p^* D'$, 
which was to be shown. \hfill \qed

\noindent \textbf{Claim 5:}  \emph{The following conditions are equivalent.}
\vspace{-0.15cm}
\begin{enumerate}
\item \emph{$\mathrm{Pic}(Y)_\mathbb{Q} = N^1(Y)_\mathbb{Q}$ and $\rho(Y) = 1$, and}
\item \emph{$X$ is $\mathbb{Q}$-factorial.}
\end{enumerate} 
\noindent\emph{Proof.}
``$\Rightarrow$": Let $D$ be a Weil divisor on $X$. By Claim 4, $D \sim p^* D'$ for some $D' \in \mathrm{Pic} Y$. Since $\rho(Y) = 1$, $D'$ is numerically proportional to $L$. Then since $\mathrm{Pic}(Y)_\Q = N^1(Y)_\Q$, $D'$ is $\Q$-linearly proportional to $L$. Hence, $D$ is $\Q$-Cartier by Claim 2.

``$\Leftarrow$": If $\mathrm{Pic}(Y)_\Q \ne N^1(Y)_\Q$, then there is a numerically trivial divisor $D$ on $Y$ which is non-torsion. It follows that there does not exist a $\lambda \in \Q$ such that $D \sim_\Q \lambda L$. An application of Claim 2 hence yields that $p^* D$ is not $\Q$-Cartier.

If $\rho(Y) \ge 2$, then there is a divisor $D$ on $Y$ which is not numerically proportional to $L$. By Claim 2 again, $p^* D$ is not $\Q$-Cartier. \hfill \qed

\noindent \textbf{Claim 6:}  \emph{We always have $\mathrm{Pic}(X)_\Q = N^1(X)_\Q$ and $\rho(X) = 1$.}

\noindent\emph{Proof.} 
For the first part, it suffices to show that any numerically trivial Cartier divisor $D$ on $X$ is torsion. By Claims 4 and 2, $D = p^* D'$ and $D' \sim_\Q \lambda L$ for some $\lambda \in \Q$. Restricting $D$ to $S_\infty$, where it coincides with $D'$, we see that $\lambda L$ is numerically trivial. Consequently, $\lambda = 0$, hence $D'$ is torsion and so is $D$.

For the second part, it suffices to show that any two Cartier divisors $D, E$ on $X$ are numerically proportional. Write $D = p^* D'$ and $E = p^* E'$. As $D' \sim_\Q \lambda L$ and $E' \sim_\Q \mu L$, the divisors $D'$ and $E'$ are $\Q$-linearly proportional, and so are $D$ and $E$. In particular, $D$ and $E$ are numerically proportional. 
\hfill \qed

Using these preparations, we can now construct the desired example.
\begin{prop}
If $Y \subset \P^3$ is a general smooth quintic, then the associated projective cone $X \subset \mathbb{P}^4$ is a Mori Dream Space, but $X$ is not log canonical, nor does it have rational singularities.
\end{prop}
\noindent\emph{Proof.} 
 By \cite[Ch.~II, Prop.~8.23]{Hartshorne}, $X$ is normal. By the Noether-Lefschetz theorem, see for example \cite{LopezNoetherLefschetz}, we have $\mathrm{Pic}(Y) \cong \Z$, so $X$ is $\Q$-factorial by Claim 5. Then by Claim 6 and Remark~\ref{rem:MDSalaHuKeel} the variety $X$ is a Mori Dream Space, since the section ring of an ample divisor is finitely generated. However, by adjunction we obtain 
\begin{equation}\label{eq:canonical}
 K_Y = \O_Y(1).
\end{equation}
Therefore, the pair $(X, \emptyset)$ is not log canonical by Claim 3. The variety $X$ doesn't have rational singularities because \[(R^2 \! f_* \O_{\tilde X})_P \cong \bigoplus_{m \geq 0} H^2(Y, \mathscr{O}_Y(m))  \supset H^2(Y, \O_Y) \cong H^0(Y, \O_Y(K_Y)) \ne 0\] by Serre duality and \eqref{eq:canonical}.  \hfill \qed

This concludes the construction of Example~\ref{ex:MDS}.
\end{ex}

\subsection*{Concluding remark}
It remains an interesting open question to characterise spaces in class $\mathscr{Q}_G$ by the existence of special tensors, some finite generation property, or similar intrinsic properties.\\
\vspace{0.4cm}

\def\cprime{$'$} \def\polhk#1{\setbox0=\hbox{#1}{\ooalign{\hidewidth
  \lower1.5ex\hbox{`}\hidewidth\crcr\unhbox0}}}
  \def\polhk#1{\setbox0=\hbox{#1}{\ooalign{\hidewidth
  \lower1.5ex\hbox{`}\hidewidth\crcr\unhbox0}}}
\providecommand{\bysame}{\leavevmode\hbox to3em{\hrulefill}\thinspace}
\providecommand{\MR}{\relax\ifhmode\unskip\space\fi MR }
\providecommand{\MRhref}[2]{%
  \href{http://www.ams.org/mathscinet-getitem?mr=#1}{#2}
}
\providecommand{\href}[2]{#2}

\vspace{0.5cm}

\begin{thebibliography}{MFK94}

\bibitem[AE12]{AlperEaston}
Jarod~D. Alper and Rob~W. Easton, \emph{Recasting results in equivariant
  geometry: affine cosets, observable subgroups and existence of good
  quotients}, Transform. Groups \textbf{17} (2012), no.~1, 1--20.

\bibitem[Art70]{ArtinAlgebraizationII}
Michael Artin, \emph{Algebraization of formal moduli. {II}. {E}xistence of
  modifications}, Ann. of Math. (2) \textbf{91} (1970).

\bibitem[B{\"a}k11]{Baeker}
Hendrik B{\"a}ker, \emph{Good quotients of {M}ori dream spaces}, Proc. Amer.
  Math. Soc. \textbf{139} (2011), no.~9, 3135--3139.

\bibitem[BB02]{BBSurvey}
Andrzej Bia{\l}ynicki-Birula, \emph{Quotients by actions of groups}, Algebraic
  quotients. Torus actions and cohomology. The adjoint representation and the
  adjoint action, Encyclopaedia Math. Sci., vol. 131, Springer, Berlin, 2002,
  pp.~1--82.

\bibitem[BBS83]{BBSommeseC*}
Andrzej Bia{\l}ynicki-Birula and Andrew~John Sommese, \emph{Quotients by {${\bf
  C}\sp{\ast}$} and {${\rm SL}(2, {\bf C}) $} actions}, Trans. Amer. Math. Soc.
  \textbf{279} (1983), no.~2, 773--800.

\bibitem[BBS85]{BBSommeseC*2}
\bysame, \emph{Quotients by {${\bf C}\sp\ast \times {\bf C}\sp \ast$} actions},
  Trans. Amer. Math. Soc. \textbf{289} (1985), no.~2, 519--543.

\bibitem[BB{\'S}92]{BBSCompleteSL2OrbitSpaces}
Andrzej Bia{\l}ynicki-Birula and Joanna {\'S}wi{\polhk{e}}cicka, \emph{On
  complete orbit spaces of {${\rm SL}(2)$} actions. {II}}, Colloq. Math.
  \textbf{63} (1992), no.~1, 9--20.

\bibitem[BB{\'S}96]{BBOpenSubsetsOfProjectiveSpacesWithQuotient}
\bysame, \emph{Open subsets of projective spaces with a good quotient by an
  action of a reductive group}, Transform. Groups \textbf{1} (1996), no.~3,
  153--185.

\bibitem[BB{\'S}97]{BBExistenceOfQuotients}
\bysame, \emph{Three theorems on existence of good quotients}, Math. Ann.
  \textbf{307} (1997), no.~1, 143--149.

\bibitem[BB{\'S}98]{BBRecipeForVectorSpaces}
\bysame, \emph{A recipe for finding open subsets of vector spaces with a good
  quotient}, Colloq. Math. \textbf{77} (1998), no.~1, 97--114.

\bibitem[CLS11]{CoxLittleSchenck}
David~A. Cox, John~B. Little, and Henry~K. Schenck, \emph{Toric varieties},
  Graduate Studies in Mathematics, vol. 124, American Mathematical Society,
  Providence, RI, 2011.

\bibitem[Cox95]{Cox}
David~A. Cox, \emph{The homogeneous coordinate ring of a toric variety}, J.
  Algebraic Geom. \textbf{4} (1995), no.~1, 17--50.

\bibitem[GH10]{ReductionInStepsBook}
Daniel Greb and Peter Heinzner, \emph{K{\"a}hlerian {R}eduction in {S}teps},
  Symmetry and {S}paces - Proceedings of a workshop in honour of {G}erry
  {S}chwarz ({E}ddy Campbell, Aloysius~G. {H}elminck, {H}anspeter {K}raft, and
  David Wehlau, eds.), Progress in Mathematics, vol. 278, Birkh\"auser, Boston,
  2010, pp.~63 -- 82.

\bibitem[GR04]{TheoryOfSteinSpaces}
Hans Grauert and Reinhold Remmert, \emph{Theory of {S}tein spaces}, Classics in
  Mathematics, Springer-Verlag, Berlin, 2004.

\bibitem[Gra63]{GrauertBemerkenswert}
Hans Grauert, \emph{Bemerkenswerte pseudokonvexe {M}annigfaltigkeiten}, Math.
  Z. \textbf{81} (1963), 377--391.

\bibitem[Gre10a]{KaehlerQuotientsGIT}
Daniel Greb, \emph{Compact {K}\"ahler quotients of algebraic varieties and
  {G}eometric {I}nvariant {T}heory}, Adv. Math. \textbf{224} (2010), no.~2,
  401--431.

\bibitem[Gre10b]{PaHq}
\bysame, \emph{Projectivity of analytic {H}ilbert and {K}\"ahler quotients},
  Trans. Amer. Math. Soc. \textbf{362} (2010), 3243--3271.

\bibitem[Gre11]{GrebAnalyticSingularities}
\bysame, \emph{Rational singularities and quotients by holomorphic group
  actions}, Ann. Sc. Norm. Super. Pisa Cl. Sci. (5) \textbf{X} (2011), no.~2,
  413 -- 426.

\bibitem[Gro65]{EGAIVPartII}
Alexander Grothendieck, \emph{\'{E}l\'ements de g\'eom\'etrie alg\'ebrique.
  {IV}. \'{E}tude locale des sch\'emas et des mor\-phis\-mes de sch\'emas.
  {II}}, Inst. Hautes \'Etudes Sci. Publ. Math. (1965), no.~24.

\bibitem[Har70]{HartshorneAmple}
Robin Hartshorne, \emph{Ample {S}ubvarieties of {A}lgebraic {V}arieties},
  Lecture Notes in Mathematics, vol. 156, Springer-Verlag, Berlin, 1970.

\bibitem[Har77]{Hartshorne}
\bysame, \emph{Algebraic {G}eometry}, Graduate Texts in Mathematics, vol.~52,
  Springer-Verlag, New York, 1977.

\bibitem[Hau09]{HausenCompleteOrbitSpaces}
J{\"u}rgen Hausen, \emph{Complete orbit spaces of affine torus actions},
  Internat. J. Math. \textbf{20} (2009), no.~1, 123--137.

\bibitem[Hau13]{HausenThreeLectures}
\bysame, \emph{Three {L}ectures on {C}ox {R}ings}, Torsors, {\'E}tale
  {H}omotopy and {A}pplications to {R}ational {P}oints, LMS Lecture Note
  Series, vol. 405, Cambridge University Press, 2013, pp.~3--60.

\bibitem[Hei89]{HeinznerFixpunktmengen}
Peter Heinzner, \emph{Fixpunktmengen kompakter {G}ruppen in {T}eilgebieten
  {S}teinscher {M}annigfaltigkeiten}, J. Reine Angew. Math. \textbf{402}
  (1989), 128--137.

\bibitem[Hei91]{HeinznerGIT}
\bysame, \emph{Geometric invariant theory on {S}tein spaces}, Math. Ann.
  \textbf{289} (1991), no.~4, 631--662.

\bibitem[HHL94]{Extensionofsymplectic}
P.~Heinzner, A.~T. Huckleberry, and F.~Loose, \emph{K\"ahlerian extensions of
  the symplectic reduction}, J. Reine Angew. Math. \textbf{455} (1994),
  123--140.

\bibitem[HK00]{HuKeel}
Yi~Hu and Sean Keel, \emph{Mori {D}ream {S}paces and {GIT}}, Michigan Math. J.
  \textbf{48} (2000), 331--348.

\bibitem[HK10]{HaconKovacs}
Christopher~D. Hacon and S{\'a}ndor~J. Kov{\'a}cs, \emph{Classification of
  higher dimensional algebraic varieties}, Oberwolfach Seminars, vol.~41,
  Birkh\"auser Verlag, Basel, 2010.

\bibitem[HL94]{ReductionOfHamiltonianSpaces}
Peter Heinzner and Frank Loose, \emph{Reduction of complex {H}amiltonian
  {$G$}-spaces}, Geom. Funct. Anal. \textbf{4} (1994), no.~3, 288--297.

\bibitem[HMP98]{SemistableQuotients}
Peter Heinzner, Luca Migliorini, and Marzia Polito, \emph{Semistable
  quotients}, Ann. Scuola Norm. Sup. Pisa Cl. Sci. (4) \textbf{26} (1998),
  no.~2, 233--248.

\bibitem[Iva10]{Ivashkovich}
Sergey Ivashkovich, \emph{Limiting behavior of trajectories of complex
  polynomial vector fields}, \texttt{arXiv:1004.2618}, 2010.

\bibitem[Kin94]{King}
Alastair~D. King, \emph{Moduli of representations of finite-dimensional
  algebras}, Quart. J. Math. Oxford Ser. (2) \textbf{45} (1994), no.~180,
  515--530.

\bibitem[Knu71]{KnutsonAlgebraicSpaces}
Donald Knutson, \emph{Algebraic {S}paces}, Lecture Notes in Mathematics, Vol.
  203, Springer-Verlag, Berlin, 1971.

\bibitem[Lop91]{LopezNoetherLefschetz}
Angelo~Felice Lopez, \emph{Noether-{L}efschetz theory and the {P}icard group of
  projective surfaces}, Mem. Amer. Math. Soc. \textbf{89} (1991), no.~438.

\bibitem[Lun73]{LunaSlice}
Domingo Luna, \emph{Slices \'etales}, Bull. Soc. Math. France, Paris, M\'emoire
  33, Soc. Math. France, Paris, 1973, pp.~81--105.

\bibitem[Lun76]{Lunaalgebraicanalytic}
\bysame, \emph{Fonctions diff\'erentiables invariantes sous l'op\'eration d'un
  groupe r\'eductif}, Ann. Inst. Fourier (Grenoble) \textbf{26} (1976), no.~1,
  ix, 33--49.

\bibitem[Mat60]{Matsushima}
Yoz{\^o} Matsushima, \emph{Espaces homog\`enes de {S}tein des groupes de {L}ie
  complexes}, Nagoya Math. J. \textbf{16} (1960), 205--218.

\bibitem[MFK94]{MumfordGIT}
David Mumford, John Fogarty, and Frances~Clare Kirwan, \emph{Geometric
  {I}nvariant {T}heory}, third ed., Ergebnisse der Mathe\-ma\-tik und ihrer
  Grenzgebiete, 2.\ Folge, vol.~34, Springer-Verlag, Berlin, 1994.

\bibitem[Mum99]{RedBook}
David Mumford, \emph{The {R}ed {B}ook of {V}arieties and {S}chemes}, Lecture
  Notes in Mathematics, vol. 1358, Springer-Verlag, Berlin, 1999.

\bibitem[Nem13]{Nemirovski}
Stefan Nemirovski, \emph{The {L}evi problem and semistable quotients}, Complex
  Var.~Elliptic Equ. \textbf{58} (2013), no.~11, 1517--1525.

\bibitem[OV90]{OnishchikVinberg}
Arkadi~L. Onishchik and Ernest~B. Vinberg, \emph{Lie {G}roups and {A}lgebraic
  {G}roups}, Springer Series in Soviet Mathematics, Springer-Verlag, Berlin,
  1990.

\bibitem[PV94]{PopovVinberg}
Vladimir Popov and Ernest~B. Vinberg, \emph{Invariant {T}heory}, Algebraic
  Geometry IV, Encyclopaedia of Mathematical Sciences, vol.~55,
  Springer-Verlag, Berlin, 1994, pp.~123--284.

\bibitem[Ros56]{Rosenlicht2}
Maxwell Rosenlicht, \emph{Some basic theorems on algebraic groups}, Amer. J.
  Math. \textbf{78} (1956), 401--443.

\bibitem[Ser56]{GAGA}
Jean-Pierre Serre, \emph{G\'eom\'etrie alg\'ebrique et g\'eom\'etrie
  analytique}, Ann. Inst. Fourier (Grenoble) \textbf{6} (1955--1956), 1--42.

\bibitem[Sha94]{ShafarevichII}
Igor~R. Shafarevich, \emph{Basic algebraic geometry. 2}, second ed.,
  Springer-Verlag, Berlin, 1994.

\bibitem[Sno82]{Snow}
Dennis~M. Snow, \emph{Reductive group actions on {S}tein spaces}, Math. Ann.
  \textbf{259} (1982), no.~1, 79--97.

\bibitem[Sum74]{completion}
Hideyasu Sumihiro, \emph{Equivariant completion}, J. Math. Kyoto Univ.
  \textbf{14} (1974), 1--28.

\end{thebibliography}
\end{document}